\documentclass[12pt]{amsart} 

\usepackage[OT2, T1]{fontenc}

\usepackage{amscd}
\usepackage{amsmath}
\usepackage{amssymb}
\usepackage{mathrsfs} 			
\usepackage{units}
\usepackage[all]{xy}

\usepackage{algorithm}
\usepackage{algorithmic}

\def\frk{\frak}               

\def\mm{{\frk m}}

\def\Phi{{\frk n}}
\def\Phi{{\frk N}}
\def\opn#1#2{\def#1{\operatorname{#2}}} 

\opn\projdim{proj\,dim} \opn\injdim{inj\,dim} \opn\rank{rank}
\opn\depth{depth} \opn\sdepth{sdepth} \opn\fdepth{fdepth}
\opn\grade{grade} \opn\height{height} \opn\embdim{emb\,dim}
\opn\codim{codim}  \opn\min{min} \opn\max{max}

\opn\Tr{Tr} \opn\bigrank{big\,rank}
\opn\superheight{superheight}\opn\lcm{lcm}
\opn\trdeg{tr\,deg}
\opn\reg{reg} \opn\lreg{lreg} \opn\ini{in} \opn\lpd{lpd}
\opn\size{size}
\opn\div{div} \opn\Div{Div} \opn\cl{cl} \opn\Cl{Cl}
\opn\Spec{Spec} \opn\Supp{Supp} \opn\supp{supp} \opn\Sing{Sing}
\opn\Ass{Ass} \opn\Min{Min}
\opn\Ann{Ann} \opn\Rad{Rad} \opn\Soc{Soc}
\opn\Im{Im} \opn\Ker{Ker} \opn\Coker{Coker} \opn\Am{Am}
\opn\Hom{Hom} \opn\Tor{Tor} \opn\Ext{Ext} \opn\End{End}
\opn\Aut{Aut} \opn\id{id}  \opn\deg{deg}

\opn\nat{nat}
\opn\pff{pf}
\opn\Pf{Pf} \opn\GL{GL} \opn\SL{SL} \opn\mod{mod} \opn\ord{ord}
\opn\Gin{Gin} \opn\Hilb{Hilb}
\opn\aff{aff} \opn\con{conv} \opn\relint{relint} \opn\st{st}
\opn\lk{lk} \opn\cn{cn} \opn\core{core} \opn\vol{vol}
\opn\link{link} \opn\star{star}
\opn\gr{gr}

\def\pot#1#2{#1[\kern-0.28ex[#2]\kern-0.28ex]}

\opn\dirlim{\underrightarrow{\lim}}
\opn\inivlim{\underleftarrow{\lim}}
\let\to=\rightarrow

\def\Implies{\ifmmode\Longrightarrow \else
        \unskip${}\Longrightarrow{}$\ignorespaces\fi}
\def\implies{\ifmmode\Rightarrow \else
        \unskip${}\Rightarrow{}$\ignorespaces\fi}
\def\iff{\ifmmode\Longleftrightarrow \else
        \unskip${}\Longleftrightarrow{}$\ignorespaces\fi}

\let\:=\colon
\newtheorem{Theorem}{Theorem}[]
\newtheorem{Lemma}[Theorem]{Lemma}
\newtheorem{Corollary}[Theorem]{Corollary}
\newtheorem{Proposition}[Theorem]{Proposition}

\theoremstyle{definition}
\newtheorem{Example}[Theorem]{Example}

\newtheoremstyle{subsection-tweak}
   {11pt}
   {3pt}%
   {}
   {}%
   {\bfseries}
   {}%
   {.5em}
   {\thmnumber{\@{#1}{}\@{#2}.}%
    \thmnote{~{\bfseries#3.}}}    

\newcounter{numberingbase}

\theoremstyle{subsection-tweak}
\newtheorem{bpp}[Theorem]{}
\newtheorem{bppt}[numberingbase]{}
\newcommand{\bbpp}{\begin{bpp}}
\newcommand{\eepp}{\end{bpp}}
\newcommand{\bbppt}{\begin{bppt}}
\newcommand{\eeppt}{\end{bppt}}

\theoremstyle{theorem}

\theoremstyle{definition}

\newcommand{\val}{\mathrm{val}}		

\let\epsilon\varepsilon
\let\phi=\varphi
\def\qed{\ifhmode\textqed\fi
      \ifmmode\ifinner\quad\qedsymbol\else\dispqed\fi\fi}
\def\textqed{\unskip\nobreak\penalty50
       \hskip2em\hbox{}\nobreak\hfil\qedsymbol
       \parfillskip=0pt \finalhyphendemerits=0}
\def\dispqed{\rlap{\qquad\qedsymbol}}

\opn\dis{dis}
\def\pnt{{\raise0.5mm\hbox{\large\bf.}}}

\opn\Lex{Lex}

\begin{document}

\title{Valuation rings of dimension one  as limits of smooth algebras}

\author{ Dorin Popescu}

\address{Simion Stoilow Institute of Mathematics of the Romanian Academy, Research unit 5, University of Bucharest, P.O. Box 1-764, Bucharest 014700, Romania, Email: {\sf dorin.popescu@imar.ro}}

\begin{abstract} As in Zariski's Uniformization Theorem we show that a valuation ring $V$  of characteristic $p>0$  of dimension one is a filtered direct limit of   smooth  ${\bf F}_p$-algebras under some conditions of transcendence degree. Under mild conditions, the algebraic immediate extensions of valuation rings are dense if they are filtered direct limit of smooth morphisms. 

{\it Key words } : Immediate extensions of valuations rings, Pseudo convergent sequences, Pseudo limits, smooth morphisms, Henselian rings.   \\
 {\it 2010 Mathematics Subject Classification: Primary 13F30, Secondary 13A18,13L05,13B40.}
\end{abstract}

\maketitle

\section*{Introduction}
Zariski predicted, and proved in characteristic $0$ in \cite{Z}, that any integral algebraic variety X
equipped with a dominant morphism $v : Spec(V) \to X$ from a valuation ring $V$ can be desingularized
along $V$ : there should exist a proper birational map ${\tilde X} \to  X$ for which the lift ${\tilde v} : Spec(V) \to {\tilde X}$
of $v$ supplied by the valuative criterion of properness would factor through the regular locus of $\tilde X$. 
This local form of resolution of singularities remains open in positive and mixed characteristic, and
implies that every valuation ring V should be a filtered direct limit of regular rings.  A filtered direct limit (in other words a filtered colimit) is a limit indexed by a small category that is filtered (see \cite[002V]{SP} or \cite[04AX]{SP}). A filtered  union is a filtered direct limit in which all objects are subobjects of the final colimit, so that in particular all the transition arrows are monomorphisms. There exists several nice extensions of Zariski's Uniformization Theorem as for example recently the result of   B. Antieau, R. Datta \cite[Theorem 4.1.1]{AD}, which says that every perfect valuation ring of characteristic $p>0$ is a filtered union of its smooth ${\bf F}_p$-subalgebras. This result is an application of \cite[Theorem 1.2.5]{T} which relies on some results from \cite{J}. Also E. Elmanto and M. Hoyois proved that an absolute integrally closed valuation ring of residue field of characteristic $p>0$ is a filtered union of its regular finitely generated  ${\bf Z}$-subalgebras (see \cite[Corollary 4.2.4]{AD}).

The goal
of this paper is to establish
a kind of Zariski's Uniformization Theorem in characteristic $p>0$ and dimension one.

\begin{Theorem}\label{T} Let $V$ be a one dimensional valuation ring containing a perfect field $F$ of characteristic $p>0$, $k$ its residue field,   $\Gamma$ its  value group and $K$ its fraction field.  Then the following statements hold
\begin{enumerate}

\item  if $k\subset V$ and $K=k(x)$ for some   system of algebraically independent  elements  $x=(x_1,\ldots,x_r)\in V^r$ over $k$ such that   $\Gamma=\oplus_{i=1}^{r-1} {\bf Z}\val(x_i)$
then
 $V$ is a filtered
 union of  its polynomial $k$-subalgebras, in particular of its smooth $F$-subalgebras,

\item if $\Gamma$ is free of rank $r$,  $k\subset V$ and  $x=(x_1,\ldots,x_r)\in V^r$ is a system of elements such that  $\Gamma=\oplus_{i=1}^r {\bf Z}\val(x_i)$ and $K/k(x)$ is algebraic  
  then $V$ is  a filtered
union of its smooth $F$-subalgebras,  in particular of its smooth ${\bf F}_p$-subalgebras,

\item if  $\Gamma$ is free of rank $r$, $k/F$ is a field extension of finite type and $\trdeg_FK=\trdeg_Fk+r$, 
then $V$ is a filtered union of its smooth $F$-subalgebras, in particular of its smooth ${\bf F}_p$-subalgebras.
\end{enumerate}
\end{Theorem}

The proof is given in Corollaries \ref{C}, \ref{co1},   \ref{co2}.  The idea of the proof of Theorem \ref{T} (1)   
is to see that $W=V\cap k(x)$ is a filtered  union of its localizations of polynomial $k$-subalgebras and then to reduce to show that  the immediate extension $W\subset V$ is  a filtered
union of its smooth $W$-subalgebras. This is done somehow in \cite[Proposition 18]{P} when  $k\supset {\bf Q}$. If char $k>0$ this does not work  (see e.g. \cite[Remark 6.10]{Po1}), the reason being that in general a valuation ring could have a so called defect as was noted by Ostrowski. The  
 Generalized Stability Theorem of F.\ V.\ Kuhlmann \cite[Theorem 1.1]{K1} says in particular that $W$ is defectless (see Corollary \ref{ku}) and so we may show (1) using Lemma \ref{l1}. For  Theorem \ref{T} (2) we see that $W\subset V$ is dense and we use a N\'eron-Schappacher Theorem (see e.g. \cite[Theorem 4.1]{Po1}) or
 a kind of General N\'eron Desingularization (see Lemma \ref{k'}) which allow us to handle dense extensions.

 When $V$ contains $\bf Q$ then an immediate algebraic extension of valuation rings $V\subset V'$  is dense by Ostrowski's Defektsatz \cite[Sect. 9, No 57]{O}. It follows that 
   $V\subset V'$ is  a filtered
direct limit of   smooth $V$-algebras (see e.g. \cite[Proposition 9]{P}). The converse is also mainly true as shows the following result (see Theorem \ref{t}).

\begin{Theorem}\label{T1}  Let  $ V'$ be an  immediate algebraic  extension  of a valuation ring $V$, 
$\hat V$ the completion of $V$ and $K$, ${\hat K}$ the fraction fields of $V$, ${\hat V}$. Assume that   $\trdeg_K{\hat K}\geq 1$,  ${\hat K}/K$ is separable,  $V'$ is a filtered direct limit of smooth $V$-algebras and either $V$ is Henselian, or $\dim V=1$. Then $V\subset V'$ is dense.
\end{Theorem}

We owe thanks to F.\ V.\ Kuhlmann   who hinted us some mistakes in some preliminary forms of the paper. Also we  owe thanks to K\k{e}stutis \v{C}esnavi\v{c}ius  who hinted us some mistakes in the preliminary forms of Theorem \ref{ky}. 
\vskip 0.5 cm

\section{The Defect}
Let $V$ be a  valuation ring with value group $\Gamma$,  $K$ its fraction field and its valuation $\val:K^*=K\setminus \{0\}\to \Gamma$. Let $F$ be a finite field extension of $K$, $v_i$, $1\leq i\leq r$ the valuations of $F$ extending $\val$ and $V_i$ be the valuation rings of $F$ defined by $v_i$. Let $e_i, f_i$ be the ramification index respectively the degree of the residue field extension of $V_i$ over $V$. It is well known that 
$$[F:K]\geq \sum_{i=1}^r e_if_i.$$
 If this inequality is equality for all finite (resp. finite separable) field  extensions $F$ of $V$ we say that $V$ is {\em defectless} (resp. {\em separably defectless}, see \cite{K} for  details and examples).     

We remind the following  important result for the existence of defectless valuation rings. 

\begin{Theorem} (\cite[Theorem 2,  VI, (8.5)]{Bou})\label{b} Let $F$ be a finite  field extension of $K$ and $B$ be the integral closure of $V$ in $F$. Then the above inequality is equality if and only if $B$ is a free $V$-module.  
\end{Theorem}

\begin{Lemma}\label{s}  Let $V$ be a  valuation ring  and $S$ a multiplicative closed system from $V$. If $V$ is defectless (resp. separably defectless) then the valuation ring $S^{-1}V$ is defectless (resp. separably defectless) too.
\end{Lemma}
\begin{proof} Assume $V$ is defectless (the proof in the separably defectless case is similar). Let $F$ be a finite field extension of the fraction field $K$ of $V$ (and of  $S^{-1}V$ too) and $B$, $B'$ be the integral closures of $V$, respectively  $S^{-1}V$ in $F$.
By Theorem \ref{b} we have to show that $B'$ is a free  $S^{-1}V$-module. As $V$ is defectless by the quoted theorem we see that $B$ is a free $V$-module and so  $S^{-1}B$ is a free  $S^{-1}V$-module. 

We claim that $B'\cong  S^{-1}B$ which is enough. Indeed, let $z\in F$ be integral over $S^{-1}V$. We have 
$$z^e+\sum_{i=0}^{e-1}(a_i/s)z^i=0$$
for some $e\in {\bf N}$, $a_i\in V$ and $s\in S$. Then
$$(sz)^e+\sum_{i=0}^{e-1}s^{e-i-1}a_i(sz)^i=0$$
and so $sz\in B$, which shows our claim.

\hfill\ \end{proof}

Next it is  very useful the following  particular form of the Generalized Stability Theorem of F. V. Kuhlmann (see \cite[Theorem 1.1]{K1}, \cite[Theorem 5.1]{K}).

\begin{Theorem}\label{k}  Let $V\subset V'$ be an extension of   valuation rings   with the same residue field, $\Gamma\subset \Gamma'$ its value group extension and $K\subset K'$ its fraction field extension.  Assume  $K'/K$ is a finite type field extension and  $\Gamma'/\Gamma$ is a finitely generated   free $\bf Z$-module  of rank   $\trdeg K'/K $. If $V$ is defectless (resp. separably defectless) then $V'$ is  defectless (resp. separably defectless) too. 
\end{Theorem}

\begin{Corollary}\label{ku}  Let $V$ be a  valuation ring of characteristic $p>0$  with  finitely generated value group $\Gamma$ and fraction field $K$. Assume  $V$ contains its residue field $k$ and $K=k(y)$ for some elements $y=(y_1,\ldots,y_r)$ of $V$ such that $\val(y_i)$, $1\leq i\leq r$ is a basis of the free $\bf Z$-module $\Gamma$.
Then $V$ is  defectless. 
\end{Corollary}
For the proof note that $k$ is defectless for the trivial valuation, $y$ is algebraically independent over $k$ by \cite[Theorem 1, (10.3)]{Bou} and apply the above theorem to $k\subset V$.

\begin{Proposition}(Kuhlmann, \cite[Theorem 1.1]{K1})  Let $V$ be a  valuation ring,  $\mm$ its maximal ideal and $X=(X_1,\ldots,X_n)$ some variables. If  $V$ is defectless  then the valuation ring $V_1=V[X]_{\mm V[X]}$ is defectless too.
\end{Proposition}
\begin{proof} Applying induction on $n$ we reduce to the case $n=1$. Let $\Gamma$ be the value group of $V$ and set $\Gamma'=\Gamma\oplus {\bf Z}\gamma$ where $\gamma$ is taken to be positive but $<$ than all positive elements of $\Gamma$ and let $w$ be the valuation on $K(X)$ extending $\val$ by $X\to \gamma$. By Theorem \ref{k} we see that $W$, the valuation ring of $w$, is defectless. Note that $\mm W$ is a prime ideal of $W$ and $W_{\mm W}\cong V_1$ is defectless by Lemma \ref{s}.
\hfill\ \end{proof}

An inclusion $V \subset V'$ of valuation rings is an \emph{immediate extension} if it is local as a map of local rings and induces equalities between the value groups and the residue fields of $V$ and $V'$.  It is dense if for any $x'\in V'$ and $\gamma\in \Gamma$ there exists $x\in V$ such that $\val(x-x')>\gamma$.

Let $\lambda$ be a fixed limit ordinal  and $v=\{v_i \}_{i < \lambda}$ a sequence of elements in $V$ indexed by the ordinals $i$ less than  $\lambda$. Then $v$ is \emph{pseudo convergent} if 

$\val(v_{i} - v_{i''} ) < \val(v_{i'} - v_{i''} )     \ \ \mbox{for} \ \ i < i' < i'' < \lambda$
(see \cite{Kap}, \cite{Sch}).
A  \emph{pseudo limit} of $v$  is an element $x \in V$ if 

$ \val(x -  v_{i}) = \val(v_{i} - v_{i'})) \ \ \mbox{for} \ \ i < i' < \lambda$. We say that $v$  is 
\begin{enumerate}
\item
\emph{algebraic} if some $f \in V[T]$ satisfies $\val(f(v_{i})) < \val(f(v_{i'}))$ for large enough $ i < i' < \lambda$;

\item
\emph{transcendental} if each $f \in V[T]$ satisfies $\val(f(v_{i})) = \val(f(v_{i'}))$ for large enough $i < i' < \lambda$.
\end{enumerate}
When for any $\gamma\in \Gamma$ it holds $\val(v_{i} - v_{i'})>\gamma$ for  $ i < i' < \lambda$
large we call $v$ {\em fundamental}.
\begin{Lemma}\label{l0}
 Let $V\subset V'$ be an immediate extension of one dimensional valuation rings and $K\subset K'$ their fraction field extension. If $V$ is separably defectless and $K'/K$ is separable and algebraic then $ K'/K$ is dense. Moreover, if $K'/K$ is not separable but  $V$ is defectless then $K'/K$ is still dense. 
 \end{Lemma}
 \begin{proof} We may reduce to the case when $K'/K$ is finite separable. Let $K^h$, $K'^h$ be the Henselizations of $K$ respectively $K'$ and $L=K^h(K')\subset K'^h$. By \cite[Theorem 2.3]{K} we have $K^h$ separably defectless and so we have $[L:K^h]=e_{L/K^h}f_{L/K^h}=1$ because there exists an unique extension of the valuation of $V$ to $K^h$. Thus $L\subset K^h$ and $K\subset K^h$ is dense because $\dim K=1$. The second statement goes similarly.
\hfill\ \end{proof}  

 \begin{Example} We consider  \cite[Example 3.1.3]{Po1} inspirated by \cite{O}. Let $k$ be a field
of characteristic $p > 0$, $X$ a variable, $\Gamma = {\bf Q}$ and $K$ the fraction field of the group
algebra $k[\Gamma]$, that is the rational function in $\{X^q\}_{q\in {\bf Q}}$. Let $P$ be the field of all formal
sums $z =
\sum_{n\in {\bf N}} a_n X^{\gamma_n}$ where $(\gamma_n)_{n\in {\bf N}}$ is a monotonically increasing sequence from
$\Gamma$ and $a_n \in k$. Set $\val(z) = \gamma_s$, where $s = \min \{n \in {\bf N} : a_n \not= 0\}$ if $ z \not= 0$ and let $V'$
be the valuation ring defined by $\val : P^*\to \Gamma$.

Let $\rho_n=(p^{n+1}-1)/(p-1)p^{n+1}$, 
$$y=-1+\sum_{n\geq 0} (-1)^nX^{\rho_n}$$
and $a_i=-1+\sum_{0\leq n\leq i} (-1)^nX^{\rho_n}$.
We have $1+\rho_n =p\rho_{n+1}$ for $n\geq 0$ and $p\rho_0=1$ and  y is a pseudo limit of the pseudo convergent
sequence $a = (a_i)_{i\in {\bf N}}$, which has no pseudo limit in $K$. Then $x$ is a root of the separable
polynomial $g=Y^p+XY+1\in K[Y]$ and the algebraic separable extension 
$V = V'\cap K\subset V'\cap K(y)$ is not dense.  
 The above lemma cannot be applied because $V$ is not separably defectless.
\end{Example} 

\begin{Lemma}\label{l1}
 Let $V\subset V'$ be an  immediate extension of one dimensional valuation rings, $K\subset K'$ their fraction field extension and $x\in V'$ a transcendental element over $K$. Assume $V$ is separably defectless and $x$ is a pseudo limit of a pseudo convergent sequence $a$ over $V$ which has no pseudo limit in $V$. Then $a$ is transcendental.
 \end{Lemma}

\begin{proof} Let $a=(a_j)_{j<\lambda}$ and assume  $a$ is algebraic. Let $h\in V[X]$ be a primitive polynomial of minimal degree among the polynomials $f\in V[X]$ such that $\val(f(a_j))<\val(f(a_{j+1})) $ for all $j<\lambda$. 
We may take $h$ separable. Indeed, assume that $h$ is not separable. We have $\val(h(x))>\val(h(a_j))$ for $j<\lambda$ large. Choose $b\in V$ such that $\val(bx)>\val(h(x))$ and set $h'=h+bX$. Note that $\val(x)=\val(a_j)$ for $j$ large, which implies $\val(ba_j)=\val(bx)>\val(h(x))>\val(h(a_j))$ for $j$ large. It follows that $\val(h'(a_j))=\val(h(a_j))$ and we may replace $h$ by $h'$, which is separable.

 By \cite[Theorem 3]{Kap} there exists a finite separable immediate extension $L=K(z)$ of $K$ such that $h(z)=0$ and  $z$ is a pseudo limit of $a$. By Lemma \ref{l0} $L$ is dense over $K$ and so $a$ has a pseudo limit in $K$ (see e.g. \cite[Lemma 2.5]{Po1}), which is false. 
\hfill\ \end{proof}

\begin{Proposition}\label{p0}
 Let $V\subset V'$ be an  immediate extension of one dimensional valuation rings and $K\subset K'$ their fraction field extension. Assume $V$ is separably defectless and $K'=K(x)$ for some $x\in K'$ which is  transcendental over $V$.  Then $V'$ is a filtered  union of localizations of its polynomial $V$-subalgebras in one variable.
 \end{Proposition}

\begin{proof} Then $x$ is a pseudo limit of a pseudo convergent sequence $a$ over $V$ which has no pseudo limit in $V$ by \cite[Theorem 1]{Kap}. But $a$ is transcendental by Lemma \ref{l1} and using \cite[Lemma 3.2]{Po} or  \cite[Lemma 15]{P} we are done.
\hfill\ \end{proof} 

\begin{Corollary} \label{C} Let $V$ be a one dimensional valuation ring, $K$ its fraction field, $k$ its residue field and $\Gamma$ its value group. Assume $k\subset V$ and $K=k(x)$ for some   system of algebraically independent  elements  $x=(x_1,\ldots,x_r)\in V^r$ over $k$ such that either $\Gamma=\oplus_{i=1}^r {\bf Z}\val(x_i)$, or  $\Gamma=\oplus_{i=1}^{r-1} {\bf Z}\val(x_i)$. Then $V'$ is a filtered  union of localizations of its polynomial $V$-subalgebras in $r$ variables.  
\end{Corollary}
\begin{proof} The first case is a consequence of \cite[Theorem 1, VI, (10.3)]{Bou} and \cite[Lemma 4.6]{Po} (see also \cite[Lemma 26 (1)]{P}). In the second case we see that $V$ is an immediate extension of $W=V\cap k(x_1,\ldots,x_{r-1})$. As above $W$ is a filtered  union of localizations of its polynomial $V$-subalgebras in $(r-1)$ variables. By Corollary \ref{ku} we have $W$ defectless and using Proposition  \ref{p0} we are done. 
\hfill\ \end{proof}

\vskip 0.5 cm

\section{Valuation rings as limits of smooth algebras}

\begin{Proposition}\label{p2}  
Let $V\subset V'$ be an  immediate extension of one dimensional valuation rings and  $K\subset K'$ their fraction field extension.   Assume $V$ is separably defectless and $K'/K$ is  algebraic separable.
 Then $V'$ is a filtered  union of its smooth $V$-subalgebras.
\end{Proposition}  
\begin{proof} 
By Lemma \ref{l0} the extension $V\subset V'$ is dense.  Then by a N\'eron-Scappacher Theorem (see e.g. \cite[Theorem 4.1]{Po1}) we see  that $V'$ is a filtered direct limit of smooth $V$-algebras. 
 \hfill\ \end{proof}

\begin{Corollary}\label{C'} Let $V$ be a one dimensional valuation ring of characteristic $p>0$  with a free (over $\bf Z$) value group $\Gamma$ of rank $r$ and fraction field $K$. Assume  $V$ contains its residue field $k$ and  $x=(x_1,\ldots,x_r)\in V^r$ is a system of elements such that  $\Gamma=\oplus_{i=1}^r {\bf Z}\val(x_i)$ and $K/k(x)$ is algebraic 
 separable. 
Then $V$ is a filtered union of its smooth $k$-subalgebras.
\end{Corollary}
\begin{proof}  By Corollary \ref{ku} we see that $W=V\cap k(x)$ is defectless and so  $V$ is a filtered  union of its smooth $W$-subalgebras using the above proposition.   But $W$ is a filtered union of its smooth $k$-algebras  (see \cite[Theorem 1, VI, (10.3)]{Bou}, \cite[Lemma 4.6]{Po},  \cite[Lemma 26 (1)]{P}), which is enough. 
 \hfill\ \end{proof}

We need the following lemma   (\cite[Lemma 7]{P} which is an extension of   
 \cite[Proposition 3]{KPP}, and \cite[Proposition 5]{P1}).

\begin{Lemma} \label{k'}
For a commutative diagram of ring morphisms

\xymatrix@R=0pt{
& B \ar[rd] & & && B \ar[dd]^{b \, \mapsto\, a} \ar[rd] & \\
A \ar[rd] \ar[ru] & & V & \mbox{that factors as follows} & A \ar[ru]\ar[rd] & & V/a^3V \\ 
& A' \ar[ru] & & && A'/a^3A' \ar[ru] &
}
\noindent with $B$ finitely presented over $A$, a $b \in B$ that is standard over $A$ (this means a special element from the ideal $H_{B/A}$ defining the non smooth locus of $B$ over $A$, for details see for example \cite[Lemma 4]{P}), and a nonzerodivisor $a \in A'$ that maps to a nonzerodivisor in $V$ that lies in every maximal ideal of $V$,
there is a smooth $A'$-algebra $S$ such that the original diagram factors as follows:

\hskip 4 cm\xymatrix@R=0pt{
& B \ar[rdd] \ar[rrd] & & \\
A \ar[rd] \ar[ru] & &  & V. \\
& A' \ar[r]  & S \ar[ru] &
}

 \end{Lemma}

In fact the separability condition is not necessary in Corollary \ref{C'}.

\begin{Corollary}\label{co1} Let $V$ be a one dimensional valuation ring of characteristic $p>0$  with a free (over $\bf Z$) value group $\Gamma$ of rank $r$ and fraction field $K$. Assume  $V$ contains its residue field $k$ and  $x=(x_1,\ldots,x_r)\in V^r$ is a system of elements such that  $\Gamma=\oplus_{i=1}^r {\bf Z}\val(x_i)$ and $K/k(x)$ is algebraic. 
Then $V$ is a filtered union of its smooth $k$-subalgebras.
\end{Corollary}
\begin{proof}
As in Corollary \ref{C'} we see that $W=V\cap k(x)$ is defectless and so the algebraic extension $W\subset V$ is dense by Lemma \ref{l0}.
Let $E\subset V$ be a finitely generated $F$-subalgebra and $w:E\to V$ its inclusion. Assume $E=F[Y]/I$, for $Y=(Y_1,\ldots,Y_n)$.   Using \cite[Lemma 1.5]{S} it is enough to show that $w$ factors through a smooth $F$-algebra. Note that $K/F$ is separable because $F$ is perfect. Thus $E/F$ is separable and
 $w(H_{E/F})\not =0$, let us assume that  $ w( H_{E/F})V \supset zV$ for some $z \in W $, $z\not =0$. Replacing $z$ by a power of it we may assume that $z=\sum_i^s b_ib'_i$ for some $b_i=\det(\partial f_{ij}/\partial Y_{j_i})$ for some systems of polynomials $f_i$ from $I$ and $b''_i\in F[Y]$ which kills $I/(f_i)$.
Similarly as in  \cite[Lemma 4]{KPP} we may
assume that we can take $s=1$, that is   for some polynomials $ f = (f_1,\ldots,f_r)$
from $I$, we have  $z\in NM E$ for some $N\in ((f):I)$ and a $r\times r$-minor $M$ of the Jacobian matrix $(\partial f_i/\partial Y_j)$ (since $V$ is a valuation ring this reduction is much easier). Thus we may assume $z$ is standard over $F$,  which is necessary later to apply Lemma \ref{k'}.

 Set $E'=E\otimes_FW$ and let $w':E'\to V$ be the map induced by $w$. We have $w'(H_{E'/W})\supset zV$ and note that 
$w'$ factors modulo $z^3$ through the smooth $W/z^3W$-algebra  $W/z^3W\cong V/z^3V$ since $W\subset V$ is dense. Then
 using Lemma \ref{k'} we see that $w'$ factors through a smooth $W$-algebra.  Since $W$ is a filtered direct limit of smooth $k$-algebra as in Corollary \ref{C'} we see that $w$ factors through a smooth $k$-algebra even $F$-algebra. Actually, we may apply as in Proposition \ref{p2} a variant of N\'eron-Schappacher Theorem to get $V$ as a filtered union of its smooth $W$-subalgebras and we are done.
 \hfill\ \end{proof}

\begin{Corollary}\label{co2} Let $V$ be a one dimensional valuation ring containing a perfect field $F$ of characteristic $p>0$,  $\Gamma$ its value group, $K$ its fraction field and $k$ its residue field. Assume $\Gamma$ is free of rank $r$, $k/F$ is a field extension of finite type and $\trdeg_FK=\trdeg_Fk+r$. 
Then $V$ is a filtered union of its smooth $F$-subalgebras, in particular of its smooth ${\bf F}_p$-algebras.
\end{Corollary}

\begin{proof} Let $E\subset V$ be a finitely generated $F$-subalgebra and $w:E\to V$ its inclusion. Using \cite[Lemma 1.5]{S} it is enough to show that $w$ factors through a smooth $F$-algebra. Using Lemma \ref{k'}  we may replace $V$ by its completion as in \cite[Proposition 9]{P}. So we may assume $V$ is Henselian.  A lifting of $k$ to $V$ could be done when $V$ is Henselian and char\ $k=p=0$ (see \cite[Theorem 2.9]{v}) but the proof goes in the same way when $p>0$ and  $k$ is separably generated over $F$. In particular the lifting could be done when $k/F$ is of finite type because $F$ is perfect. Thus we may assume $k\subset V$ and $\trdeg_kK=r$. Choose some elements  $x=(x_1,\ldots,x_r)\in V^r$  such that  $\Gamma=\oplus_{i=1}^r {\bf Z}\val(x_i)$. Then $W=V\cap k(x)\subset V$ is algebraic and we may proceed as in Corollary \ref{co1}.  
 \hfill\ \end{proof}

 \section{Algebraic pseudo convergent sequences}

 Let  $ V'$ be an  immediate extension  of a valuation ring $V$, $K\subset K'$ their fraction field extension, $v= (v_j)_{j<\lambda}$  a pseudo-convergent sequence in $V$ which is not fundamental and has  a pseudo limit $x$ in  $ V'$, but having no pseudo limit in $K$.  Suppose that  $K'=K(x)$.  We need the following result \cite[Theorem 1.8]{KC}.
 
 \begin{Theorem}(Kuhlmann-\' Cmiel)\label{kc}  Assume that $x$ is not a root of an irreducible polynomial $f\in K[X]$ and either  $\dim V=1$, or $V$ is Henselian. Then 
for all $z\in V$ there is a  $\nu<\lambda$ such that $\val(f(z))<\val(f(v_j))$ for all $\nu<j<\lambda$.
 \end{Theorem}

\begin{Theorem} \label{ky}  Assume  $x$ is transcendental over $K$ and  either  $\dim V=1$, or $V$ is Henselian.
 Then the following statements  are equivalent:
 \begin{enumerate}
 
 \item for every  polynomial $f\in V[X]$ with $f(x)\not =0$ there exists a $y\in V$ such that $\val(f(y))= \val(f(x))$, 
\item for every  polynomial $f\in V[X]$ with $f(x)\not =0$   there exists a $y\in V$ such that $\val(f(y))\geq \val(f(x))$
 \item $v$ is transcendental,
 \item $V'$ is a filtered  union of its localizations of polynomial $V$-subalgebras in one variable.
 \end{enumerate}
 
\end{Theorem}

\begin{proof} First assume that (4) holds and let $f$ be as in (1).  Let $d\in V$, $z\in V'$ be such that $f(x)=dz$ and $zz'=1$ for some $z'\in V'$. Then the solution $x,z,z'$ in $V'$ of the polynomials $F_1=f-dZ, F_2=ZZ'-1 \in V[X,Z,Z']$ must be contained in a localization of a polynomial $V$-subalgebra $C=V[u']_{P(u')}$ of $V'$, where $P\in V[U]$ and  $u'\in V'$ is  transcendental over $V$. Choose a $u\in V$ such that $u\equiv u'$ modulo $\mm'$ the maximal ideal of $V'$. Then $P(u)$ is a unit in $V$ and the map $\rho:C\to V$ given by $u'\to u$ is a retraction of $V\subset C$ and $y=\rho(x), \rho(z), \rho(z')$ is a solution of $F_1,F_2$ in $V$, and so $\val(f(y))=\val(d)=\val(f(x))$.

Clearly (2) follows from (1) and assume that (2) holds but $v$ is algebraic. This is not possible as F.\ V.\ Kuhlmann said, the proof being done in Theorem \ref{kc} (see also \cite[Lemma 5.4]{KV}). Let  $h\in V[T]$ be a polynomial with  minimal degree among the polynomials $f\in V[T]$ such that
 $\val(f(v_i))<\val(f(v_j))$ for large $i<j<\lambda$. Then $h$ is irreducible (see the proof of \cite[(II,4), Lemma 12]{Sch}). Note that $\val(h(x)-h(v_j))\geq \val(x-v_j)=\val(v_{j+1}-v_j)$. As the last term is increasing we must have $\val(h(x))\geq \val(h(v_j))$ for all $j$.   By Theorem \ref{kc} for every $z\in V$ we have $\val(h(z))<\val(h(v_j))$ for $j$ large. Thus 
   $h$ fails the condition (2) and so $v$ cannot be algebraic.

If $v$ is transcendental then $V'$ is  a filtered  union of its localizations of polynomial $V$-subalgebras in one variable by \cite[Lemma 3.2]{Po} (see also \cite[Lemma 15]{P}).
\hfill\ \end{proof}

 \begin{Corollary} In the assumptions of the above theorem, if $v$ is algebraic  then $V'$ is not  a filtered  union of its localizations of polynomial $V$-subalgebras in one variable.
 \end{Corollary}
 
 \begin{Lemma} \label{r2}  In the assumptions of the above theorem, assume that $V$ is Henselian and  there exists an immediate extension of valuation rings $V'\subset V_1$ such that $V_1$ is a filtered direct limit of smooth $V$-algebras. Then   for every  polynomial $f\in V[X]$ such that $f(x)\not =0$   there exists a $y\in V$ such that $\val(f(y))= \val(f(x))$.
\end{Lemma}
\begin{proof}  Fix $f\in V[X]$. Let $d\in V$, $z\in V'$ be such that $f(x)=dz$ and $zz'=1$ for some $z'\in V'$. Since $V_1$ is a filtered direct limit of smooth $V$-algebras, the solution $x,z',z'$ in $V'$ of the polynomials $F_1=f-dZ, F_2=ZZ'-1 \in V[X,Z,Z']$ must  come from  a solution $\bar x$, $\bar z$, ${\bar z}'$ in a smooth $V$-algebra $C$. But $V$ is Henselian and so there exists  a retraction  $\rho:C\to V$ of $V\subset C$. Thus $y=\rho({\bar x}), \rho({\bar z}), \rho({\bar z}')$ is a solution of $F_1,F_2$ in $V$, and so $\val(f(y))=\val(d)=\val(f(x))$.
\hfill\ \end{proof}

\begin{Lemma} \label{r3}  In the assumptions of the above theorem, assume that $\dim V=1$ and  there exists an immediate extension of valuation rings $V'\subset V_1$ such that $V_1$ is a filtered direct limit of smooth $V$-algebras. Let  $\hat V$  be the completion of $V$ and $K$, $\hat K$ the fraction fields of $V$,  $\hat V$. Suppose that ${\hat K}/K$ is separable.  Then   for every  polynomial $f\in V[X]$ with $f(x)\not=0$ there exists a $y\in V$ such that $\val(f(y))= \val(f(x))$, 
\end{Lemma}
\begin{proof}  Let     ${\hat V}_1$ be the completion of $V_1$ and $K_1$ the fraction field of $V_1$. Then $ {\hat V}_1\cap {\hat K}={\hat V}$. Note that  $K_2={\hat K}(V_1)$ is a localization of ${\hat K}\otimes_K K_1$ since ${\hat K}/K$ is separable. Then $V_2={\hat V}_1\cap K_2$ is  a localization of  ${\hat V}\otimes_V V_1$ and so $V_2$ is a filtered direct limit of smooth $\hat V$-algebras. As $\hat V$ is Henselian we get $\val(f(z))=\val(f(x))$ for some $z\in {\hat V}$ using the above lemma. Choose $y\in V$ such that $\val(y-z)>\val(f(x))$. Then $\val(f(y))=\val(f(x))$.
\hfill\ \end{proof}

\begin{Theorem}\label{t}  Let  $ V'$ be an  immediate algebraic  extension  of a valuation ring $V$, 
$\hat V$ the completion of $V$ and $K$, ${\hat K}$ the fraction fields of $V$, ${\hat V}$. Assume that   $\trdeg_K{\hat K}\geq 1$,  ${\hat K}/K$ is separable, there exists an immediate extension of valuation rings $V'\subset V_1$ such that $V_1$ is a filtered direct limit of smooth $V$-algebras and either $V$ is Henselian, or $\dim V=1$. Then $V\subset V'$ is dense.
\end{Theorem}

 \begin{proof}
 Assume that $V\subset V'$ is not  dense and let $x\in V'$ which is not in $\hat V$. Then $x$ is a pseudo limit of a pseudo convergent sequence $v$ which is not fundamental and has no pseudo limits in $V$ by \cite[Theorem 1]{Kap}. Note that $v$ is algebraic because $x$ is algebraic over $K$.  Choose a transcendental $t\in {\hat V}$ over $K$. Multiplying $t$ with a constant of $V$ of high enough value we may assume that $x+t$ is still a pseudo limit of $v$. 

 Let     ${\hat V}_1$ be the completion of $V_1$  and set $V_2={\hat V}_1\cap {\hat K}(V_1)$. As in the proof of the above lemma we see that $V_2$ is a  filtered direct limit of smooth $\hat V$-algebras. Thus  for every  polynomial $f\in V[X]$ such that $f(x)\not =0$   there exists a $y\in V$ such that $\val(f(y))= \val(f(x+z))$ using Lemmas \ref{r2}, \ref{r3}.  By Theorem \ref{ky} applied to $x+z$ we get $v$ transcendental, which is false. 
\hfill\ \end{proof}


\vskip 0.5 cm
\begin{thebibliography}{99}


\bibitem{AD} B.\ Antieau and R.\ Datta, Valuation rings are derived splinters,  {\em arxiv/AG:2002.010627v1}.



\bibitem{Bou}   N.\ Bourbaki,   \'El\'ements de math\'ematique. Alg\`ebre commutative, chap. I-VII, Hermann (1961, 1964,
1965); chap. VIII-X, Springer, (2006, 2007), (French).


\bibitem{J} A.\ J.\ de Jong, Smoothness, semi-stability and alterations, {\em Inst. Hautes Études Sci. Publ. Math.}, {\bf 83},
51–93, (1996).


\bibitem{Kap} I.\ Kaplansky,  Maximal fields with valuations, {\em Duke Math. J.}, {\bf 9}, 303-321, (1942).


\bibitem{KPP} Z.\ Kosar, G.\ Pfister, D.\ Popescu,  Constructive N\'eron Desingularization of algebras with big smooth locus,  {\em Communications in Algebra}, {\bf 46},  1902-1911, (2018), {\em  arXiv/AC:1702.01867}.

\bibitem{K1} F.\ V. \ Kuhlmann, Elimination of Ramification I: The Generalized Stability Theorem, {\em Trans. AMS}, {\bf 362},  5697-5727,  (2010), {\em arXiv/AC:1003.5678}.

\bibitem{K} F.\ V. \ Kuhlmann,  The Defect, {\em arXiv/AC:1004.2135v1}.

\bibitem{KC} F.\ V.\   Kuhlmann, H.\ \'Cmiel, Observations and conjectures on key polynomials, Preprint 2020.




\bibitem{KV} F.\ V.\   Kuhlmann, I.\ Vlahu, The relative approximation degree in valued function fields, {\em Math. Z.}, {\bf 276}, 203–235, (2014).



\bibitem{O} A.\  Ostrowski,  Untersuchungen zur arithmetischen Theorie der K\"orper, {\em Math. Z.}, {\bf 39},
321-404, (1935).


 \bibitem{Po} D.\ Popescu,  On Zariski's uniformization theorem, in Algebraic geometry, Bucharest 1982 (Bucharest,
1982), {\em Lecture Notes in Math.}, {\bf 1056}, Springer, Berlin,  264-296,  (1984).
 
 \bibitem{Po1} D.\ Popescu, Algebraic extensions of valued fields, {\em J. Algebra}, {\bf 108}, 
513-533, (1987).
\bibitem{P1} D.\ Popescu, Simple General Neron Desingularization in local $\bf Q$-algebras, {\em Communications in Algebra}, {\bf 47},  923 - 929, (2019),{\em arXiv:AC/1802.05109}.  
  \bibitem{P} D.\ Popescu,  N\'eron desingularization of extensions of valuation rings with an Appendix by K\k{e}stutis \v{C}esnavi\v{c}ius, in Proceedings of the conference 'Transient Transcendence in Transylvania’ (Trans’19: https://specfun.inria.fr/bostan/trans19/),  Eds. Alin Bostan, Kilian
Raschel,
to appear in a special volume of the Springer collection PROMS (Proceedings in Mathematics and Statistics, https://www.springer.com/series/10533),  {\em arxiv/AC:1910.09123v4}.

  
  
  
  
\bibitem{Sch} O.\ F.\ G.\ Schilling,  The theory of valuations, {\em Mathematical Surveys, American Math. Soc.}, (1950).

\bibitem{SP} A.\ J.\ de Jong et al., The Stacks Project. Available at \\ {\em http://stacks.math.columbia.edu}.

 \bibitem{S} R.\ Swan,  Neron-Popescu desingularization, in "Algebra and Geometry", Ed. M. Kang, {\em International Press}, Cambridge,  135-192,  (1998).

\bibitem{T} M.\ Temkin,  Tame distillation and desingularization by p-alterations, {\em Ann. of Math. }, {\bf 186},  97–126, (2017).

\bibitem{v}  L.\ van den Dries,  Lectures on the model theory of valued fields., in  Model
theory in algebra, analysis and arithmetic, {\em Lect. Notes in
Math.}, {\bf 2111}, Springer, Heidelberg,  55–157, (2014).

\bibitem{Z}  O.\ Zariski,  Local uniformization on algebraic varieties, {\em Ann. of Math.}, {\bf 41}, 852-896, (1940). 
  

\end{thebibliography}
\end{document}